\documentclass[12pt,a4paper,reqno]{amsart}
\usepackage{amssymb}
\usepackage{amscd}
\usepackage{enumerate}
\usepackage{graphicx}
\usepackage{siunitx}
\usepackage{tikz-cd}
\usepackage{color}

\usepackage[all]{xy}           
\usepackage{graphicx}
\usepackage{color}

\theoremstyle{plain}

\numberwithin{equation}{section}
\newtheorem{theo}{Theorem}[section]

\newtheorem{lem}[theo]{Lemma}
\newtheorem{cor}[theo]{Corollary}
\newtheorem{defi}[theo]{Definition}

\theoremstyle{definition}
\newtheorem{rem}[theo]{Remark}

\def\ot{\otimes}
\def\om{\omega}
\def\lan{\langle}
\def\ran{\rangle}
\def\al{\alpha}

\def\de{\delta}
\def\pa{\partial}

\def\la{\lambda}

\def\k{{\bf k}}
\def\c{{\bf c}}

\def\sl{{\mathfrak{sl}}}
\newcommand{\E}{{\mathcal{E}}}

\newcommand{\R}{{\mathcal R}}

\newcommand{\cH}{{\mathcal H}}

\newcommand{\cL}{{\mathcal L}}

\newcommand{\W}{\mathcal W}

\def\bZ{{\mathbb Z}}

\def\bC{{\mathbb C}}

\def\mh{\mathfrak{h}}

\def\mg{\mathfrak{g}}

\def\Res{\mbox{\rm Res}}

\def\End{\text{\rm End}}

\def\id{\text{\rm id}}

\def\Der{\mbox{\rm Der}}
\def\NO{\mbox{\,$\circ\atop\circ$}\,}

\begin{document}
\title[]
{Representations for  three-point Lie algebras of genus zero }

\author{Dong Liu}
\address{Department  of Mathematics, Huzhou Teachers College, Zhejiang Huzhou, 313000, China}
\email{liudong@zjhu.edu.cn}
\author{Yufeng Pei}
\address{Department of Mathematics, Shanghai Normal University,
Shanghai, 200234, China} \email{pei@shnu.edu.cn}

\author[Xia]{Limeng Xia}

\address{Institute of Applied System Analysis, Jiangsu University, Jiangsu Zhenjiang, 212013,
China}\email{xialimeng@ujs.edu.cn}


\date{}

\begin{abstract}
In this paper, we study  representations for  three-point  Lie algebras of genus
zero  based on the Cox-Jurisich's presentations. We construct two functors which transform simple restricted modules with nonzero levels over the standard affine algebras into simple modules over the three-point affine algebras of genus
zero. As a corollary, vertex  representations are constructed for the three-point affine algebra of genus
zero using vertex operators. Moreover, we construct a Fock module for certain quotient of three-point Virasoro algebra of genus
zero.
\end{abstract}
\maketitle

\section{Introduction}
Krichever-Novikov (KN) algebras were introduced and studied by Krichever and Novikov \cite{KN87a,KN87b,KN89} via meromorphic objects on compact Riemann
surfaces of arbitrary genus with two marked points where poles are allowed.  This generalizes the classical Virasoro and the affine Lie algebras, which correspond to the geometric situation of genus zero with possible poles only at $\{0,\infty\}$.
Schlichenmaier \cite{S90a,S90b,S90c} extended the whole setup to the multipoint case for arbitrary genus. KN algebras are used to construct analogues of important mathematical objects used in string theory. Moreover structure and representation theory have been developed for general KN algebras, see \cite{S2014} and references therein. The n-point affine Lie algebras which are a type of Krichever-Novikov algebras
of genus zero also appeared in the work of Kazhdan and Lusztig \cite{KL93} and in the description of the conformal blocks \cite{FBZ01}. Bremner explicitly described the universal central extensions of these algebras in \cite{Bre94a}. The n-point Virasoro algebras  were studied in \cite{CGLZ14}, which are natural generalizations of the classical Virasoro algebra and have as quotients multipoint genus zero KN algebras. Their universal central extensions and some other
algebraic properties were also determined. The modules of densities for the n-point Virasoro algebras were constructed and investigated. Vertex algebras associated to certain related KN type algebras have been studied in \cite{LS11,PS16}.

The three-point Lie algebras of genus zero are probably the simplest nontrivial examples beyond  the affine Kac-Moody algebras and the Virasoro algebra. Recall the three-point ring $R$ denotes the ring of rational
functions with poles only in the set $\{a_1, a_2, a_3\}$, which  is isomorphic to $\bC[s,s^{-1}, (s-1)^{-1}]$. Schlichenmaier has a different description of the
three-point ring using $\bC[(z^2-a^2)^k,z(z^2-a^2)^k\mid k\in\bZ]$ where $a\neq0$. Cox and Jurisich \cite{CJ14} observed that the three-point ring $R$ is isomorphic to $\bC[t,t^{-1},u\mid u^2=t^2+4t]$. The three-point affine Kac-Moody algebras are the universal central extensions of the three-point loop algebras $\mg\ot R$, where $\mg$ is a finite dimensional simple Lie algebra. Following
 from the work of Schlichenmaier \cite{S2014}, every three-point affine Kac-Moody algebra has  a two-dimensional center.
Benkart and Terwilliger \cite{BT07} determined the universal central extension of the three-point loop algebra $\sl(2)\ot R$ in
terms of the tetrahedron algebra. Ito and Terwilliger  described the finite-dimensional
irreducible  representations for $\sl(2)\ot R$.
Free field realizations for the three-point affine algebras associated to $\sl(2)$ have been constructed by  Cox and Jurisich \cite{CJ14}. The three-point Virasoro algebra is the universal central extension of the three-point Witt algebra $\Der R$.
In recent papers \cite{CJM16,CJM18}, Ben Cox et al.,  gave  Fock modules and  the Segal-Sugawara construction for the three-point Virasoro algebra on the Fock space of the three-point affine algebra of type $A_1$. Free field representations of the elliptic algebras \cite{Bre94b} and the 4-point algebras \cite{Bre95}  have been studied in \cite{BCF09,Cox08}.

As shown by Schlichenmaier \cite{S2003}  one has to consider an almost-grading, i.e., a notion of positive and negative operators if one wants to construct representations of the three-point affine  algebra. Because there is up to equivalence and rescaling only one local cocycle compatible with the almost-grading. This means from the representation theory side that the three-point affine  algebra should be the corresponding algebra with one-dimensional center. Moreover, considering the Segal-Sugawara construction for the three-point affine algebra, an almost-grading is also needed \cite{SS}.

In this paper, we study  representations for the three-point Lie algebras of genus
zero. We will use Cox and Jurisich's presentation of the three-point ring. It is convinient to compute defining relations  when  we construct
explictly representations of the three-point algebras. We present natural realizations of the three-point affine algebras in term of the infinite-dimensional Lie algebras related to the ring of formal Laurent series in $t^{\pm1}$. Then we construct two functors which transform simple restricted modules with nonzero levels over the standard affine algebras into simple modules over the three-point affine algebras.  To the best of our knowledge, these simple modules for  the three-point affine algebras are new. As a corollary, certain  simple modules of nonzero levels for the three-point affine algebras associated to $\sl(2)$ can be realized via the standard vertex operators. Furthermore, using  the  standard Heisenberg algebra,  we  construct a Fock representation for certian quotient of  three-point Virasoro algebra. This can be viewed as an analogue of the classical Virasoro construction. Note that the general Segal-Sugawara construction of KN type algebras for every number of points and any genus  was shown by Schlichenmaier and Sheinman \cite{SS}.	

Our approach  can be applied to study representations of the hyperelliptic  algebras \cite{Cox16} and the superelliptic  algebras \cite{CGLZ17}, which are natural generalizations of the n-point algebras. We will discuss this topic in future paper.

Our paper is organized as follows.  In Section 2, we construct irreducible representations for the three-point affine Lie algebras and give  vertex operator realizations of the three-point affine Lie algebras. In Section 3, we  construct a Fock representation for the three-point Virasoro algebra.

Throughout the paper, we shall use $\bC$ and $\bZ$ to denote the sets of complex numbers and integers respectively. All of algebras are over the complex number
field $\bC$.

\section{Three-point affine algebras}
In this section, we construct two functors which transform simple restricted modules with nonzero levels over the standard affine algebras into simple modules over the three-point affine algebras.

Let $\mg$ be a  Lie algebra with a
non-degenerate symmetric invariant bilinear form $\lan\cdot, \cdot\ran$.  The  affine Lie algebra associated to $\mg$, denoted by $\hat\mg$, is the universal central
extension of the Lie algebra $\mg\ot\bC[t,t^{-1}]$:
$$
\hat\mg=\mg\ot\bC[t,t^{-1}]\oplus\bC\k
$$
with the bracket relations:
\begin{eqnarray*}
{[a(m), b(n)]}&=&[a,b](m+n)+m\lan a,b\ran\de_{m+n,0}\k,\quad [\k,\hat\mg]=0,
\end{eqnarray*}
for $a,b\in\mg,,m,n\in\bZ$, where $a(n)=a\ot t^m$.

\begin{defi}
If $W$ is a $\hat\mg$-module on which $\k$ acts as a complex
scalar $\ell$,  we say that $W$ is of level $\ell$.  We define a category $\E^+$ to consist of $\hat\mg$-modules $W$ if for every $a\in\mg$
and $w\in W$, $a(n)w = 0$ for $n$ sufficiently large. Similarly, we also define a category $\E^-$ to consist of $\hat\mg$-modules $W$ if for every $a\in\mg$
and $w\in W$, $a(n)w = 0$ for $n$ sufficiently small.
\end{defi}

Denote by $\R$ the
quotient algebra of $\bC[t, t^{-1}, u]$ modulo relation
$$
u^2=p(t):=t^2+4t.
$$
Let $\mg^1$ be a vector
space isomorphic to $\mg$, with a fixed linear isomorphism $a\in\mg\mapsto a^1\in \mg^1$.

We begin by recalling three-point affine Lie algebras of genus zero.
\begin{defi}[\cite{CJ14}]\label{3a}  The three-point affine Lie algebra of genus zero associated to $\mg$, denoted by $\hat\mg_p$, is the universal central
extension of the Lie algebra $\mg\ot\R$:
$$
\hat\mg_p=\mg\ot\bC[t,t^{-1}]\oplus\mg^1\ot\bC[t,t^{-1}]\oplus\bC\k_+\oplus\bC\k_{-}
$$
with the bracket relations:
\begin{eqnarray*}
{[a(m), b(n)]}&=&[a,b](m+n)+m\lan a,b\ran\de_{m+n,0}(\k_++\k_{-}),\\
{[a(m), b^1(n)]}&=&[a^1(m),b(n)]=[a,b]^1(m+n)+m\lan a,b\ran\la_{m+n+1}4^{m+n+1}\k_{-},\\
{[a^1(m), b^1(n)]}&=&4[a,b](m+n+1)+[a,b](m+n+2)\\
                &&+\left((4m+2)\de_{m+n+1,0}+(m+1)\de_{m+n+2,0}\right)\lan a,b\ran(\k_++\k_{-}),\\
{[\hat\mg_p,\k_\pm]}&=&[\k_\pm,\hat\mg_p]=0,
\end{eqnarray*}
for $a,b\in\mg, a^1,b^1\in \mg^1,m,n\in\bZ$,
where $a(n)=a\ot t^m,a^1(m)=a^1\ot t^m$,
\begin{eqnarray}
\lambda_0=1, \lambda_n=\frac{\frac12\cdot(\frac12-1)\cdots(\frac12-(n-1))}{n!}.
\end{eqnarray}
\end{defi}

\begin{rem}
Note that we have slightly modified the original relations in \cite{CJ14} by replacing the
central elements $\om_0,\om_1$ of $\hat\mg_p$ to $-(\k_++\k_{-}),-\k_{-}$, respectively.
\end{rem}

\begin{rem}
For $\mg$ a simple Lie algebra, $\hat\mg_p$ is the three-point affine Kac-Moody algebra of genus zero; for $\mg$ abelian,
$\hat\mg_p$ is the three-point Heisenberg algebra of genus zero.
\end{rem}

To better study representations of $\hat\mg_p$, we shall make use of the following
affine Lie algebras $\tilde\mg^\pm=\mg\ot\bC((t^{\pm1}))\oplus\bC\k$ in \cite{FBZ01}, defined by
$$
[a\ot f, b\ot g]=[a,b]\ot fg+\lan a,b\ran \Res_{t} (f'g)\k,\quad [\k, \tilde\mg^\pm]=[\tilde\mg^\pm,\k]=0
$$
for $a,b\in\mg, f,g\in \bC((t^{\pm1}))$, where
$
\Res_{t} f$
denotes the coefficient of $t^{-1}$ in $f$.


\begin{lem}\label{emb}
	\begin{itemize}
\item[(1)]Let
$
\rho_+:  \hat\mg_p\hookrightarrow \tilde\mg^{+},
$
be  a linear map defined
by
\begin{eqnarray*}
	\rho_+(a(n))&=&a\otimes t^{2n},\\ \rho_+(a^1(n))&=& 2\sum_{i=0}^\infty \frac{\lambda_i}{4^i}a\otimes t^{2n+2i+1},\\
	\rho_+(\k_+)&=& 2\k,\\ \rho_+(\k_-)&=&0,
\end{eqnarray*}for $a\in\mg, n\in\bZ$.
Then $\rho_+$ is  a homomorphism of Lie algebras
\item[(2)] Let
$
\rho_-:  \hat\mg_p\hookrightarrow \tilde\mg^{-}
$  be a  linear map defined by
\begin{eqnarray*}
	\rho_-(a(n))&=&a\otimes t^{n},\\ \rho_-(a^1(n))&=& \sum_{i=0}^\infty {\lambda_i}{4^i}a\otimes t^{n-i+1},\\
	\rho_-(\k_+)&=& 0,\\ \rho_-(\k_-)&=&\k
\end{eqnarray*}
for $a\in\mg, n\in\bZ$.	Then $\rho_-$ is a homomorphisms of Lie algebras.	
\end{itemize}

\end{lem}

\begin{proof}
(1) There is a canonical embedding of the three-point ring:
$$
\iota_+: \R\to \bC((t^{\frac{1}{2}})).
$$
For $f\in \R$ , $\iota_+f$ is the expansion of $f$ as a formal Laurent series in $t^{\frac{1}{2}}$. In particular,
\begin{eqnarray*}
\sqrt{p(t)}=2t^\frac12 \sqrt{1+\frac{t}{4}} &=&2t^\frac12\sum_{n=0}^\infty\frac{\lambda_n}{4^n}t^n\in\bC((t^\frac12)).
\end{eqnarray*}

 For $a,b\in\mg, m,n\in\bZ$, we have
\begin{eqnarray*}
&&{[\rho_+(a(m)), \rho_+(b(n))]}=[a,b]\ot t^{2(m+n)}+2m\lan a,b\ran\de_{m+n,0}\k=\rho_+[a(m),b(n)],
\end{eqnarray*}
\begin{eqnarray*}
{[\rho_+(a(m)), \rho_+(b^1(n))]}&=&[a,b]\ot 2\sum_{i=0}^\infty \frac{\lambda_i}{4^i} t^{2m+2n+2i+1}+2m\lan a,b\ran \Res_t 2\sum_{i=0}^\infty \frac{\lambda_i}{4^i}t^{2m+2n+2i} \k \\
	&=&[a,b]\ot 2\sum_{i=0}^\infty \frac{\lambda_i}{4^i} t^{2m+2n+2i+1}\\
	&=&\rho_+[a(m),b^1(n)],
	\end{eqnarray*}
\begin{eqnarray*}
	{[\rho_+(a^1(m)), \rho_+(b^1(n))]}&=&[a,b]\ot \left(2\sum_{i=0}^\infty \frac{\lambda_i}{4^i} t^{2m+2i+1}\right)\left(2\sum_{j=0}^\infty \frac{\lambda_j}{4^j} t^{2n+2j+1}\right)\\
	&&+\lan a,b\ran \Res_t \left(\sqrt{p(t^2)}t^{2m}\right)'\sqrt{p(t^2)}t^{2n} \k\\
	&=&[a,b]\ot (t^4+4t^2)t^{2m+2n}\\
	&&+((4m+2)\de_{m+n+2,0}+(m+1)\de_{m+n+2,0})(2\k)\\
	&=&\rho_+[a^1(m),b^1(n)].
\end{eqnarray*}
It follows that $\rho_+$ is a homomorphism of Lie algebras.

(2) Consider the following canonical embeddings of the three-point ring:
$$
\iota_-: \R\to \bC((t^{-1})).
$$
For $f\in \R$, $\iota_{-}f$ is
its expansion as a formal Laurent series in $t^{-1}$. In particular,
\begin{eqnarray*}
\sqrt{p(t)}=t\sqrt{1+\frac{4}{t}} &=&t\sum_{n=0}^\infty \lambda_n{4^n}{t^{-n}}\in\bC((t^{-1})).
\end{eqnarray*}

 For $a,b\in\mg, m,n\in\bZ$, we have

\begin{eqnarray*}
	{[\rho_-(a(m)), \rho_-(b(n))]}&=&[a,b]\ot t^{m+n}+m\lan a,b\ran\de_{m+n,0}\k\\
	&=&\rho_-[a(m),b(n)],
\end{eqnarray*}
\begin{eqnarray*}
{[\rho_-(a(m)), \rho_-(b^1(n))]}&=&[a,b]\ot \sum_{i=0}^\infty {\lambda_i}{4^i} t^{m+n-i+1}+m\lan a,b\ran \Res_t \sum_{i=0}^\infty {\lambda_i}{4^i}t^{m+n-i} \k \\
&=&[a,b]\ot \sum_{i=0}^\infty {\lambda_i}{4^i} t^{m+n-i+1}+m\lan a,b\ran  {\lambda_{m+n+1}}{4^{m+n+1}} \k \\
&=&\rho_-[a(m),b^1(n)],
\end{eqnarray*}
\begin{eqnarray*}
	{[\rho_-(a^1(m)), \rho_-(b^1(n))]}&=&[a,b]\ot \left(\sum_{i=0}^\infty {\lambda_i}{4^i} t^{m-i+1}\right)\left(\sum_{j=0}^\infty {\lambda_j}{4^j} t^{n-j+1}\right)\\
	&&+\lan a,b\ran \Res_t \left(\sqrt{p(t)}t^{m}\right)'\sqrt{p(t)}t^{n} \k\\
	&=&[a,b]\ot (t^2+4t)t^{m+n}+((4m+2)\de_{m+n+2,0}+(m+1)\de_{m+n+2,0})\k\\
	&=&\rho_-[a^1(m),b^1(n)].
\end{eqnarray*}
Then  $\rho_-$ is a homomorphism of Lie algebras.
\end{proof}

\begin{defi}
If $W$ is a $\hat\mg_{p}$-module on which $\k_+$ and $\k_{-}$ act as complex
scalars $\ell_+,\ell_{-}$, respectively, we say that $W$ is of level $(\ell_+,\ell_{-})$.
\end{defi}

Now we state the main result in this section.

\begin{theo}\label{th1} For $\ell_\pm\in \bC$, let $V^\pm$ be irreducible $\hat\mg$-modules of level $\ell_\pm$ in $\E^\pm$. Then
  \begin{itemize}
  \item[(i)] $V^+$ is an irreducible $\hat\mg_p$-module of level $(2\ell_+,0)$, with
  \begin{eqnarray*}
a(n)&\mapsto& a\otimes t^{2n},\\
a^1(n)&\mapsto& 2\sum_{i=0}^\infty \frac{\lambda_i}{4^i}a\otimes t^{2n+2i+1},\\
\k_+&\mapsto& 2\ell_+,\\
\k_-&\mapsto&0.
\end{eqnarray*}
  \item[(ii)] $V^-$ is an irreducible $\hat\mg_p$-module of level $(0,\ell_-)$, with
  \begin{eqnarray*}
a(n)&\mapsto& a\otimes t^{n},\\
a^1(n)&\mapsto& \sum_{i=0}^\infty {\lambda_i}{4^i}a\otimes t^{n-i+1},\\
\k_+&\mapsto& 0,\\
\k_-&\mapsto&\ell_-.
\end{eqnarray*}
  \end{itemize}
\end{theo}
\begin{proof}
(i) Suppose that $V$ is an irreducible $\hat\mg$-module of level $\ell_+$ in $\E^+$, then it is  is a natural $\tilde\mg$-module of level $(\ell_+,0)$. It follows that $V$ is a $\hat\mg_p$-module of level $(2\ell_+,0)$.

For any $n\in \bZ$ and $a\in\mg$, we have $a(2n)\in\hat\mg_p$.
For any $v\in V$, there exists $N$ such that $a(n)v=0$ for $a\in\mg$ and $n>N$. It follows that there exists a polynomial $q(t)\in\bC[t,t^{-1}]$ such that
$$
2q(t^2)\sum_{i=0}^\infty \frac{\lambda_i}{4^i}t^{2n+2i+1}=t^{2n+1}+\sum_{i>N}d_it^i,
$$
and
$$
a(2n+1)v=\rho_{+}(a^1\otimes q(t)t^n)\cdot v.
$$
Hence $V$ remains irreducible as a $\hat\mg_p$-module.

(ii) Let $V$ be an irreducible $\hat\mg$-module of level $\ell_{-}$ in $\E^-$. Then  $V$ is a  $\tilde\mg$-module of level $(0,\ell_{-})$. By Lemma \ref{emb},  $V$ is a $\hat\mg_p$-module of level $(0,\ell_{-})$. Since $\hat\mg\hookrightarrow \hat\mg_p$,  $V$ remains irreducible as a $\hat\mg_p$-module.
\end{proof}

Next we present vertex  representations  for the three-point affine algebras in terms of  vertex operators. We will restrict ourselves to the type $A_1$. For arbitrary type $A,D,E$  it is straightforward.

Let $Q=\bZ\al $ be the root
lattice associated to $\mg=\sl(2)$ with the Chevalley basis $\{x_+,x_-, h\}$. In addition, we normalize the invariant form $\lan\cdot ,\cdot \ran$ so that $\lan\al,\al\ran=2$. Consider the associated Heisenberg algebra  $\hat\mh=\bigoplus_{n\in\bZ}\bC \al(m)\oplus\k$ with
the bracket relations given by
\begin{eqnarray*}
[\al(m),\al(n)]=2m\delta_{m+n,0}\k,\quad [\hat\mh,\k]=0.
\end{eqnarray*}

The polynomial ring $P=\bC[y_i\mid i=1,2,\cdots]$ can be equipped with the structure of an $\hat\mh$-module by defining
$$
\al(0)f=0,\quad \al(n)=2n\frac{\pa}{\pa y_n}(f),\quad \al(-n)f=y_nf,\quad\k f=f,\quad\forall f\in P, n>0,
$$
yielding the so-called Fock space representation of $\hat\mh$.

Let $\bC[Q]=\oplus_{n\in\bZ}\bC e_{n\al}$ denote the group algebra associated to $Q$.
On $V_{Q}=P\otimes\bC[Q]$, we define the following operators by
\begin{eqnarray*}
\alpha(-m)\cdot(f\otimes e_{n\alpha})&=&y_{m}f\otimes e_{n\alpha},\\
\alpha(m)\cdot(f\otimes e_{n\alpha})&=&2m\frac{\partial}{\partial y_m}(f)\otimes e_{n\alpha},\\
\alpha(0)\cdot(f\otimes e_{n\alpha})&=&2nf\otimes e_{n\alpha},\\
e^{k\alpha}\cdot(f\otimes e_{n\alpha})&=&f\otimes e_{(k+n)\alpha},\\
\k(f\otimes e_{n\alpha})&=&f\otimes e_{n\alpha},
\end{eqnarray*}
for all $f\in P,  n,k\in\bZ$ and $m\in\bZ_+$, and
\begin{eqnarray*}
z^{k\alpha}(v\otimes e^{n\alpha})&=&z^{2kn}v\otimes e_{n\alpha},\\
X(\pm\alpha,z)&=&\exp\Big(\mp\sum_{n=1}^\infty\frac{\alpha(-n)}{n}z^{n}\Big)
\exp\Big(-\sum_{n=1}^\infty\frac{\alpha(n)}{n}z^{-n}\Big)e_{\pm\alpha}z^{\pm\alpha+1}\\
&=&\sum_{n\in\bZ}X(\pm\alpha)_nz^{-n},\end{eqnarray*}
where $X(\pm\alpha)_n\in \End (V_{Q})$. It is known that  $V_{Q}$ can be equipped with the structure of a highest weight $\hat\sl(2)$-module of level $1$ by defining
\begin{eqnarray*}
&&x_\pm(n)\mapsto X(\pm\alpha)_n,\quad h(n)\mapsto\alpha(n),\quad\k\mapsto\id_{V_{Q}},
\end{eqnarray*}
yielding the so-called vertex operator representation of $\hat\sl(2)$ (cf. \cite{LL03}). Similarly,  $V_{Q}^-=V_{Q}$ is  a lowest weight $\hat\sl(2)$-module of level $-1$ by defining
\begin{eqnarray*}
&&x_\pm(n)\mapsto X(\mp\alpha)_{-n},\quad h(n)\mapsto -\alpha(-n),\quad\k\mapsto -\id_{V_{Q}}.
\end{eqnarray*}

\begin{cor}
\begin{itemize}
\item[(i)]$V_{Q}$ is a module of the three-point affine algebra $\sl(2)_p$  given by
\begin{eqnarray*}
x_\pm(n)&\mapsto& X(\pm\alpha)_{2n},\\
h(n)&\mapsto&\alpha(2n),\\
x_\pm^1(n)&\mapsto& 2\sum_{i=0}^\infty \frac{\lambda_i}{4^i}X(\pm\alpha)_{2n+2i+1},\\
h^1(n)&\mapsto& 2\sum_{i=0}^\infty \frac{\lambda_i}{4^i}\alpha(2n+2i+1),\\
\k_+&\mapsto& 2\id_{V(Q)},\\
\k_-&\mapsto& 0.
\end{eqnarray*}
\item[(ii)]$V_{Q}^-$ is a module of the three-point affine algebra $\sl(2)_p$ given by
\begin{eqnarray*}
x_\pm(n)&\mapsto& X(\mp\alpha)_{-n},\\
h(n)&\mapsto&-\alpha(-n),\\
x_\pm^1(n)&\mapsto& \sum_{i=0}^\infty {\lambda_i}{4^i}X(\mp\alpha)_{-n+i-1},\\
h^1(n)&\mapsto& \sum_{i=0}^\infty {\lambda_i}{4^i}\alpha(-n+i-1),\\
\k_+&\mapsto& 0,\\
\k_-&\mapsto& -\id_{V(Q)}.
\end{eqnarray*}
\end{itemize}
\end{cor}

\section{A Fock module over certain quotient of the three-point Virasoro  algebra of genus zero}
In this section,  we construct a Fock module over certain quotient three-point Virasoro algebra of genus zero.

\begin{defi}[\cite{CGLZ14},\cite{CJM16},\cite{S2017}]The  three-point Witt algebra of genus zero is the derivation Lie algebra of the three-point ring $\R$, given by
$$
\W_p=\Der\,\R=\bigoplus_{n\in\bZ}\bC d_n\oplus \bigoplus_{n\in\bZ}\bC e_n
$$
with the bracket relations:
\begin{eqnarray*}
[d_m,d_n]&=&(m-n)\left(4d_{m+n+1}+d_{m+n+2}\right),\\
{[d_m,e_n]}&=&(4m-4n+2)e_{m+n+1}+(m-n+1)e_{m+n+2},\\
{[e_m,e_n]}&=&(m-n)d_{m+n}
\end{eqnarray*}
for $m,n\in\bZ$.
\end{defi}

\begin{defi}\label{QA}\label{3v}Let $\cL_p$ be the one-dimensional central extension of $\W_p$:
$$
\cL_p=\W_p\oplus\bC \c
$$
where $\c$ is central and the nontrivial bracket
relations are written in terms of generating functions  $$
d(z)=\sum d_n z^{-n-2},\quad e(x)=\sum e_n z^{-n-2},
$$
as follows.
\begin{eqnarray*}
{[d(z),d(w)]}&=&(p(x_{2})d(w))'z^{-1}\delta\left(\frac{w}{z}\right)+2p(w)d(w)\frac{\partial}{\partial w}z^{-1}\delta\left(\frac{w}{z}\right)\\
&&+\frac{1}{12}p^{2}(w)\left(\frac{\partial}{\partial w}\right)^{3}z^{-1}\delta\left(\frac{w}{z}\right)\c
+\frac{1}{4}p(w)p'(w)\left(\frac{\partial}{\partial w}\right)^{2}z^{-1}\delta\left(\frac{w}{z}\right)\c\\
&&+\frac{1}{8}p(w)p''(w)\frac{\partial}{\partial w}z^{-1}\delta\left(\frac{w}{z}\right)\c +\frac{1}{16}(p'(w))^2\frac{\partial}{\partial w}z^{-1}\delta\left(\frac{w}{z}\right)\c\\
{[e(z),e(w)]}&=&d'(w)z^{-1}\delta\left(\frac{w}{z}\right)+2d(w)\frac{\partial}{\partial w}z^{-1}\delta\left(\frac{w}{z}\right)\\
&&+p(w)\left(\frac{\partial}{\partial w}\right)^{3}z^{-1}\delta\left(\frac{w}{z}\right)\c+\frac{3}{2}p'(w)\left(\frac{\partial}{\partial w}\right)^{2}z^{-1}\delta\left(\frac{w}{z}\right)\c,\\
{[d(z),e(w)]}&=&(p(w)e(w))'z^{-1}\de\left(\frac{w}{z}\right)+2p(w)e(w)\frac{\pa}{\pa w}z^{-1}\de\left(\frac{w}{z}\right)\\
&&+\frac{1}{2}p'(w)e(w)z^{-1}\de\left(\frac{w}{z}\right),
\end{eqnarray*}
where $\delta(\frac{w}{z})=\sum_{n\in\bZ}\left(\frac{w}{z}\right)^n$.

\end{defi}

\begin{rem}

The three-point Virasoro algebra of genus zero is the universal central extension of $\W_p$. It has two dimensional center. It is clear that the algebra $\cL_p$ is a quotient  of the three-point Virasoro algebra of genus zero.
\end{rem}

\begin{rem}
The 2-cocycle of $\W_p$  in Definition \ref{QA} is slightly different from one of 2-cocycles obtained by Cox and Jurisich \cite{CJ14}.
\end{rem}
\begin{defi}
If  $W$ is an $\cL_{p}$-module on which $\c$ acts as a complex
scalar $c$, we say that $W$ is of central charge $c$.
\end{defi}

Let $\mathcal{H}=\bC h$ be a 1-dimensional abelian Lie algebra
equipped with the following symmetric bilinear form $\lan\cdot ,\cdot\ran$ such that $\lan h,h\ran=1$. Let $\hat\cH=\cH\ot\bC[t,t^{-1}]\oplus\bC\k$ be the corresponding affine Lie algebra with
$$
[h(m),h(n)]=m\de_{m+n,0}\k,\quad\forall m,n\in\bZ,
$$
where $h(m)=h\ot t^{m}$, and $\k$ is central.

Let $V$ be an $\hat\cH$-module of level $\frac{1}{2}$ in $\E^+$. From Theorem \ref{th1} (i), we have
\begin{eqnarray*}
&&[H(m),H(n)]=m\de_{m+n,0},\\
&&[H^1(m),H^1(n)]=\left((4m+2)\de_{m+n+1,0}+(m+1)\de_{m+n+2,0}\right),\\
&&[H(m),H^1(n)]=[H^1(m),H(n)]=0,
\end{eqnarray*}
for $m,n\in\bZ$, where
$$
H(m)=h(2m),\quad H^1(m)=2\sum_{i=0}^\infty \frac{\lambda_i}{4^i}h(2m+2i+1)\in\End (V).
$$

 Set
$$
H(z)=\sum_{n\in\bZ}H(n)z^{-n-1},\quad H^1(z)=\sum_{n\in\bZ}H^1(n)z^{-n-1}\in (\End V)[[z,z^{-1}]].
$$

 We define the normal-ordering operation $\NO\ \NO$ by
\begin{eqnarray*}
 && \NO h(m)h(n)\NO=\begin{cases}
   h(m)h(n)&m<0,\\
   h(n)h(m)&m\geq0.
  \end{cases}
  \end{eqnarray*}
This induces the following contraction of two fields $X(z)$, $Y(w)$ by
$$
\underbrace{X(z)Y(w)}=X(z)Y(w)-\NO X(z)Y(w)\NO,
$$
for $X,Y\in \{h, H,H^1\}$.

Then we have the following lemma.

\begin{lem}\label{L1}
\begin{eqnarray*}
\underbrace{H(z)H(w)}&=& \frac{1}{(z-w)^2},\\
\underbrace{H^1(z)H^1(w)}&=& \frac{p(w)}{(z-w)^2}+\frac{\frac{1}{2}p'(w)}{z-w}.\end{eqnarray*}
\end{lem}
\begin{proof}
It is straightforward.	
\end{proof}

Set
\begin{eqnarray*}
D(z)&=& \frac{1}{2}\left(\NO H^1(z)H^1(z)\NO+p(z)\NO H(z)H(z)\NO\right),\quad E(z)= \NO H^1(z)H(z)\NO.
\end{eqnarray*}

The following lemmas  give the
operator product expansions of normally ordered products of $D(z)$,$E(z)$, $H(z)$, and $H^1(z)$.

\begin{lem}\label{L2}
\begin{eqnarray*}
 \underbrace{D(z)H(w)}&=&\frac{p(w)H(w)}{(z-w)^2}+\frac{p'(w)H(w)+p(w)a'(w)}{z-w},\\
 \underbrace{D(z)H^1(w)}&=&\frac{p(w)H^1(w)}{(z-w)^2}+\frac{p(w)(H^1(w))'+\frac{1}{2}H^1(w)p'(w)}{z-w},\\
 \underbrace{E(z)H(w)}&=&\frac{H^1(w)}{(z-w)^2}+ \frac{(H^1(w))'}{z-w},\\
 \underbrace{E(z)H^1(w)}&=&\frac{p(w)H(w)}{(z-w)^2}+\frac{p(w)a'(w)}{z-w}+\frac{\frac{1}{2}p'(w)H(w)}{z-w}.
\end{eqnarray*}

\end{lem}
\begin{proof} Using the well-known Wick's theorem \cite{Ka97}, we have
\begin{eqnarray*}
 \underbrace{D(z)H(w)}&=&\frac{1}{2}\underbrace{\NO H^1(z)H^1(z)\NO H(w)}+\frac{1}{2}p(z)\underbrace{\NO H(z)H(z)\NO H(w)}\\
 &=&\frac{p(w)H(w)}{(z-w)^2}+\frac{p'(w)H(w)+p(w)a'(w)}{z-w},
 \end{eqnarray*}

\begin{eqnarray*}
 \underbrace{D(z)H^1(w)}&=&\frac{1}{2}\underbrace{\NO H^1(z)H^1(z)\NO H^1(w)}+\frac{1}{2}p(z)\underbrace{\NO H(z)H(z)\NO H^1(w)}\\
 &=&\frac{p(w)H^1(w)}{(z-w)^2}+\frac{p(w)(H^1(w))'+\frac{1}{2}H^1(w)p'(w)}{z-w},
 \end{eqnarray*}

\begin{eqnarray*}
 \underbrace{E(z)H(w)}&=&\underbrace{\NO H^1(z)H(z)\NO H(w)}\\
 &=&\frac{H^1(w)}{(z-w)^2}+ \frac{(H^1(w))'}{z-w},
\end{eqnarray*}

\begin{eqnarray*}
\underbrace{E(z)H^1(w)} &=&\underbrace{\NO H^1(z)H(z)\NO H^1(w)}\\
 &=&\frac{p(w)H(w)}{(z-w)^2}+\frac{p(w)a'(w)}{z-w}+\frac{\frac{1}{2}p'(w)H(w)}{z-w}.
\end{eqnarray*}

\end{proof}

\begin{lem}\label{L3}
\begin{eqnarray*}
 \underbrace{D(z)D(w)}&=&\frac{p^2(w)}{(z-w)^4}+\frac{p(w)p'(w)}{(z-w)^3}+\frac{\frac{1}{4}p(w)p''(w)+\frac{1}{8}p'^2(w)}{(z-w)^2}
 +\frac{\frac{1}{12}p(w)p'''(x)}{z-w}\\
 &&+\frac{2p(w)D(w)}{(z-w)^2}+\frac{p'(w)D(w)+p(w)D'(w)}{z-w}.
\end{eqnarray*}
\end{lem}
\begin{proof}

\begin{eqnarray*}
 &&\frac{1}{4}\underbrace{\NO H^1(z)H^1(z)\NO \NO H^1(w)H^1(w)\NO}\\
 &=&\frac{2}{4}\left(\frac{p(w)}{(z-w)^2}+\frac{\frac{1}{2}p'(w)}{z-w}\right)^2
 +\left(\frac{p(w)}{(z-w)^2}+\frac{\frac{1}{2}p'(w)}{z-w}\right)\NO H^1(z)H^1(w)\NO\\
 &=&\frac{\frac{1}{2}p^2(w)}{(z-w)^4}+\frac{\frac{1}{2}p(w)p'(w)}{(z-w)^3}+\frac{\frac{1}{8}p'^2(w)}{(z-w)^2}\\
 &&+\frac{p(w)\NO H^1(w)H^1(w)\NO}{(z-w)^2}+\frac{\frac{1}{2}p'(w)\NO H^1(w)H^1(w)\NO+p(w)\NO H^1(w)(H^1(w))'\NO}{z-w}.
\end{eqnarray*}

\begin{eqnarray*}
 &&\frac{1}{4}p(z)p(w)\underbrace{\NO H(z)H(z)\NO \NO H(w)H(w)\NO}\\
 &=&\frac{1}{2}\frac{p(z)p(w)}{(z-w)^4}
 +\frac{p(z)p(w)}{(z-w)^2}\NO H(z)H(w)\NO\\
 &=&\frac{1}{2}\frac{p(w)}{(z-w)^4}\left(p(w)+p'(w)(z-w)+\frac{p''(w)}{2}(z-w)^2+\frac{p'''(w)}{6}(z-w)^3+\cdots\right)\\
&& +\frac{p(w)}{(z-w)^2}\NO \left(p(w)H(w)+(p(w)H(w))'(z-w)+\cdots\right)H(w)\NO\\
 &=&\frac{\frac{1}{2}p^2(w)}{(z-w)^4}+\frac{\frac{1}{2}p(w)p'(w)}{(z-w)^3}+\frac{\frac{1}{4}p(w)p''(w)}{(z-w)^2}
 +\frac{\frac{1}{12}p(w)p'''(x)}{z-w}\\
 &&+\frac{p^2(w)\NO H(w)H(w)\NO}{(z-w)^2}+\frac{p(w)p'(w)\NO H(w)H(w)\NO+p^2(w)\NO H(w)H'(w)\NO}{z-w}.
\end{eqnarray*}

\begin{eqnarray*}
 &&D(z)D(w)\\
 &=& \frac{1}{4}\underbrace{\NO H^1(z)H^1(z)\NO \NO H^1(w)H^1(w)\NO}+\frac{1}{4}p(z)p(w)\underbrace{\NO H(z)H(z)\NO \NO H(w)H(w)\NO}\\
 &=&\frac{p^2(w)}{(z-w)^4}+\frac{p(w)p'(w)}{(z-w)^3}+\frac{\frac{1}{4}p(w)p''(w)+\frac{1}{8}p'^2(w)}{(z-w)^2}
 +\frac{\frac{1}{12}p(w)p'''(x)}{z-w}\\
 &&+\frac{2p(w)\left(\NO \frac{1}{2}H^1(w)H^1(w)+\frac{1}{2}p(w)\NO H(w)H(w)\NO\right)}{(z-w)^2}\\
 &&+\frac{p'(w)\left(\frac{1}{2}\NO H^1(w)H^1(w)\NO+\frac{1}{2}p(w)\NO H(w)H(w)\NO\right)}{z-w}\\
 &&+\frac{p(w)\left(\frac{1}{2}\NO H^1(w)H^1(w)\NO+\frac{1}{2}p(w)\NO H(w)H(w)\NO\right)'}{z-w}\\
 &=&\frac{p^2(w)}{(z-w)^4}+\frac{p(w)p'(w)}{(z-w)^3}+\frac{\frac{1}{4}p(w)p''(w)+\frac{1}{8}p'^2(w)}{(z-w)^2}
 +\frac{\frac{1}{12}p(w)p'''(x)}{z-w}\\
 &&+\frac{2p(w)D(w)}{(z-w)^2}+\frac{p'(w)D(w)+p(w)D'(w)}{z-w}.
\end{eqnarray*}

\end{proof}

\begin{lem}\label{L4}
\begin{eqnarray*}
 \underbrace{D(z)E(w)}&=&\frac{2p(w)E(w)}{(z-w)^2}+\frac{\frac{3}{2}p'(w)E(w)+p(w)e'(w)}{z-w}.
\end{eqnarray*}
\end{lem}

\begin{proof}

\begin{eqnarray*}
 && \underbrace{D(z)E(w)}\\
 &=&\frac{1}{2}\underbrace{\NO H^1(z)H^1(z)\NO \NO H^1(w)H(w)\NO}+\frac{1}{2}p(z)\underbrace{\NO H(z)H(z)\NO \NO H^1(w)H(w)\NO}\\
 &=&\left(\frac{p(w)}{(z-w)^2}+\frac{\frac{1}{2}p'(w)}{z-w}\right)\NO H^1(z)H(w)\NO+\frac{1}{(z-w)^2}\NO p(z)H(z)H^1(w)\NO\\
 &=&\left(\frac{p(w)}{(z-w)^2}+\frac{\frac{1}{2}p'(w)}{z-w}\right)\NO \left(H^1(w)+(H^1)'(w)(z-w)\right)H(w)\NO\\
 &&+\frac{1}{(z-w)^2}\NO \left(p(w)H(w)+(p(w)H(w))'(z-w)\right)H^1(w)\NO\\
&=&\frac{p(w)\NO H^1(w)H(w)\NO}{(z-w)^2}+\frac{\frac{1}{2}p'(w)\NO H^1(w)H(w)\NO+p(w)\NO (H^1)'(w)H(w)\NO}{z-w}\\
 &&+\frac{\NO p(w)H(w)H^1(w)\NO}{(z-w)^2}+\frac{\NO (p(w)H(w))'H^1(w)\NO}{z-w}\\
&=&\frac{2p(w)\NO H^1(w)H(w)\NO}{(z-w)^2}+\frac{\frac{3}{2}p'(w)\NO H^1(w)H(w)\NO+p(w)\left(\NO H^1(w)H(w)\NO\right)'}{z-w}\\
&=&\frac{2p(w)E(w)}{(z-w)^2}+\frac{\frac{3}{2}p'(w)E(w)+p(w)E'(w)}{z-w}.
\end{eqnarray*}

\end{proof}

\begin{lem}\label{L5}
\begin{eqnarray*}
\underbrace{E(z)E(w)} &=&\frac{p(w)}{(z-w)^4}+\frac{\frac{1}{2}p'(w)}{(z-w)^3}+\frac{2D(w)}{(z-w)^2}+\frac{D'(w)}{z-w}.
\end{eqnarray*}
\end{lem}
\begin{proof}
\begin{eqnarray*}
 &&\underbrace{E(z)E(w)}\\
 &=&\underbrace{\NO H^1(z)H(z)\NO\NO H^1(w)H(w)\NO}\\
  &=&\left(\frac{p(w)}{(z-w)^2}+\frac{\frac{1}{2}p'(w)}{z-w}\right)\frac{1}{(z-w)^2}\\
  &&+\left(\frac{p(w)}{(z-w)^2}+\frac{\frac{1}{2}p'(w)}{z-w}\right)\NO H(z)H(w)\NO+\frac{1}{(z-w)^2}\NO H^1(z)H^1(w)\NO\\
&=&\frac{p(w)}{(z-w)^4}+\frac{\frac{1}{2}p'(w)}{(z-w)^3}+\frac{p(w)\NO H(w)H(w)\NO+\NO H^1(w)H^1(w)\NO}{(z-w)^2}\\
&&+\frac{p(w)\NO H(w)a'(w)\NO+\frac{1}{2}p'(w)\NO H(w)H(w)\NO}{z-w}+\frac{\NO H^1(w)(H^1(w))'\NO}{z-w}\\
&=&\frac{p(w)}{(z-w)^4}+\frac{\frac{1}{2}p'(w)}{(z-w)^3}\\
 &&+\frac{2\left(\frac{1}{2}\NO H^1(w)H^1(w)\NO+\frac{1}{2}p(w)\NO H(w)H(w)\NO\right)}{(z-w)^2}\\
&&+\frac{\left(\frac{1}{2}\NO H^1(w)H^1(w)\NO+\frac{1}{2}p(w)\NO H(w)H(w)\NO\right)'}{z-w}\\
&=&\frac{p(w)}{(z-w)^4}+\frac{\frac{1}{2}p'(w)}{(z-w)^3}+\frac{2D(w)}{(z-w)^2}+\frac{D'(w)}{z-w}.
\end{eqnarray*}
\end{proof}

We recall the general operator product expansions in  \cite{Ka97}. Suppose $X(z),Y(w)$
are two fields such that
$$
\underbrace{X(z)Y(w)}=\sum_{j=0}^{N-1}Z^j(w)\frac{1}{(z-w)^{j+1}},
$$
then
$$
[X(z),Y(w)]=\sum_{j=0}^{N-1}Z^j(w)\frac{1}{j!}\left(\frac{\pa}{\pa w}\right)^jz^{-1}\de\left(\frac{w}{z}\right).
$$
where $N$ is a positive integer and $Z^j(w)\in \End(V)[[w,w^{-1}]]$.

Now we state the main result in this section.

\begin{theo}Let $V$ be an $\hat\cH$-module of level $\frac{1}{2}$ in $\E^+$. Then  $V$ is a representation of the Lie algebra $\cL_p$ of central charge $2$, with
\begin{eqnarray*}
d(z)&\mapsto& \frac{1}{2}\left(\NO H^1(z)H^1(z)\NO+p(x)\NO H(z)H(z)\NO\right),\\
e(z)&\mapsto& \NO H^1(z)H(z)\NO,\\
\c&\mapsto& 2.
\end{eqnarray*}
\end{theo}
\begin{proof}Let
$$
d(z)\mapsto D(z),\quad e(z)\mapsto E(z),\quad  \c\mapsto 2.
$$
It follows immediately from Lemma \ref{L1}--Lemma \ref{L5}.
\end{proof}

\vskip15pt

 \centerline{\bf ACKNOWLEDGMENTS}

 We gratefully acknowledge the partial financial support from the NNSF (Nos.11971315, 11871249, 11771142),  and the Jiangsu Natural Science Foundation(No.BK20171294). We thank referee for his/her helpful comments and suggestions.

\end{document}